\documentclass[12pt, reqno]{amsart}
\usepackage{amsmath, amsthm, amscd, amsfonts, amssymb, graphicx, color}
\usepackage[bookmarksnumbered, colorlinks, plainpages]{hyperref}

\newtheorem{theorem}{Theorem}[section]
\newtheorem{lemma}[theorem]{Lemma}

\newtheorem{corollary}[theorem]{Corollary}
\theoremstyle{definition}

\theoremstyle{remark}
\newtheorem{remark}[theorem]{Remark}
\numberwithin{equation}{section}

\begin{document}
\setcounter{page}{1}


\title[On improved Young inequality  ]{On improvement of Young inequality using the Kontrovich constant}
\author[  M. Khosravi and  A. Sheikhhosseini ]{Maryam Khosravi$^1$   Alemeh
Sheikhhosseini$^2$}

\address{$^{1}$ Department of
 Pure Mathematics, Faculty of Mathematics and Computer,
 Shahid Bahonar University of Kerman,
Kerman, Iran} \email{khosravi$_-$m @uk.ac.ir;
khosravi$_-$m2000@yahoo.com }

\address{$^{2}$ Department of
 Pure Mathematics, Faculty of Mathematics and Computer,
 Shahid Bahonar University of Kerman,
Kerman, Iran} \email{sheikhhosseini@uk.ac.ir;
hosseini8560@gmail.com }


\begin{abstract}
Some improvements of Young inequality and its reverse
for positive numbers with  Kontrovich constant are given. Using these inequalities some
operator versions and Hilbert-Schmidt norm versions for matrices
 are proved.
\end{abstract}
\maketitle
\section{introduction}
Let $a,b$ be two positive number. The famous Young inequality
states that
$$a^{1-\nu}b^{\nu}\leq(1-\nu)a+\nu b,$$
for every $0\leq\nu\leq1$. By defining weighted arithmetic and
geometric means as
$$a\nabla b=(1-\nu)a+\nu b,\qquad a\sharp_{\nu}
b=a^{1-\nu}b^{\nu},$$ we can consider the Young inequality as
weighted arithmetic-geometric means inequality. This inequality
has received an increasing attention in the literature.

One of the best improvement of Young inequality, was obtained by
 F. Kittaneh and Y. Manasrah \cite{k1}, as follows:
  \begin{equation*}\label{re1}
  a\sharp_{\nu}  b+ r(\sqrt{a}-\sqrt{b})^{2} \leqslant a\nabla_{\nu} b
  \end{equation*}
  where $ r=\min \{\nu, 1-\nu\} $  and $ s=\max \{\nu, 1-\nu \}. $\\
  The authors of \cite{oh} obtained another  refinement of the  Young inequality
   as follows:
    \begin{equation*}\label{re2}
  (a\sharp_{\nu}  b)^{2} + r^{2}(a-b)^{2}    \leqslant    ( a\nabla_{\nu} b )^{2},
    \end{equation*}
  where  $ r=\min  \{\nu, 1-\nu   \}.$

In \cite{wz}, the authors obtained another  improvement
    of  the  Young inequality and its reverse as follows:
 \begin{equation}\label{e03}
    K (\sqrt{h}, 2)^{r'}a\sharp_{\nu}  b
  \leqslant a\nabla_{\nu} b-r (\sqrt{a}-\sqrt{b})^{2} ,
  \end{equation}
  and
\begin{equation}\label{e03'}
 a\nabla_{\nu} b-R (\sqrt{a}-\sqrt{b})^{2} \leqslant   K (\sqrt{h},
   2)^{-r'}a\sharp_{\nu}  b,
  \end{equation}
  where $h=\frac{b}{a}$ and $K(t,1)=\frac{(1+t)^2}{4t}$ is the
  Kontrovich constant,
  $ r=\min \{ \nu, 1- \nu \}$, $R=\max\{\nu,1-\nu\}$ and $ r'=\min \lbrace  2r, 1-2r
  \rbrace$.

In addition, with the same notations as above, another type of the
reverse of Young inequality
  using Kontrovich constant is as follows: \cite{tai}
\begin{equation}\label{e03''}
a\nabla_{\nu} b-r (\sqrt{a}-\sqrt{b})^{2} \leqslant   K (\sqrt{h},
   2)^{R'}a\sharp_{\nu}  b,
  \end{equation}
  where  $ R'=\max \lbrace  2r, 1-2r
  \rbrace$.

Note that the $K(t,2)\geq1$ for all $t>0$ and attains its minimum
at $t=1$. Also $K(t,2)=K(\frac{1}{t},2)$.

Recently,  Liao and  Wu \cite{jmi} obtained the following
refinement of  inequality
  (\ref{e03}) and \eqref{e03'}:
  \begin{align}\label{1.7}
 a\nabla_{\nu} b \geqslant&  \nu (\sqrt{a}-\sqrt{b})^{2}+
 r ((ab)^{\frac{1}{4}}-\sqrt{a})^{2} + K (
  h^{\frac{1}{4}}, 2)^{r_{1}}a\sharp_{\nu}  b,\\
a\nabla_{\nu} b \leqslant&  (1-\nu) (\sqrt{a}-\sqrt{b})^{2}- r
((ab)^{\frac{1}{4}}-\sqrt{b})^{2} + K (
  h^{\frac{1}{4}}, 2)^{-r_{1}}a\sharp_{\nu}  b,\nonumber
  \end{align}
for $  0 < \nu \leqslant \frac{1}{2},$
  and
  \begin{align}\label{1.8}
 a\nabla_{\nu} b \geqslant& (1- \nu) (\sqrt{a}-\sqrt{b})^{2}+
 r ((ab)^{\frac{1}{4}}-\sqrt{b})^{2} + K
 (h^{\frac{1}{4}}, 2)^{r_{1}}a\sharp_{\nu}  b,\\
 a\nabla_{\nu} b \leqslant & \nu (\sqrt{a}-\sqrt{b})^{2}-
 r ((ab)^{\frac{1}{4}}-\sqrt{a})^{2} + K
 (h^{\frac{1}{4}}, 2)^{-r_{1}}a\sharp_{\nu}  b,\nonumber
  \end{align}
for $  \frac{1}{2}< \nu < 1,$ where  $r=\min \{2(1-\nu), 1-
2(1-\nu)  \}$
    and  $ r_{1} =\min \lbrace  2r, 1-2r\rbrace. $

For more related inequalities see \cite{bakh,k3,kn,salemi,hl}.

These numerical inequalities, leads to similar operator
inequalities. For this purpose, let $\mathbb{B}(\mathbb{H})$ stand
for the $ C^{*}$-algebra of all bounded linear operators on a
complex Hilbert space $ \mathbb{H} $. An operator $A\in
\mathbb{B}(\mathbb{H})$ is called self-adjoint if $A=A^*$,
positive ( and is denoted by $A\geq0$) if $A$ is self-adjoint with
non-negative spectrum and strictly positive if $A$ is an
invertible positive operator.

If $\mathbb{H}$ is finite dimensional, of dimension $n$, then we
identify $\mathbb{B}(\mathbb{H})$ with $\mathbb{M}_{n} $ of all $
n \times n $ complex matrices. In this case, we use the terms
positive semidefinite and positive definite matrices, instead of
positive and strictly positive operators, respectively.

The partial order $A\leq B$, on the class of self-adjoint
operators, means that $B-A$ is a positive operator.

The weighted arithmetic and geometric mean for strictly positive
operators $A,B$, is defined by
$$ A \nabla_{\nu} B=(1-\nu) A+ \nu B,\qquad A\sharp_{\nu}B=A^{\frac{1}{2}}
(A^{-\frac{1}{2}}BA^{-\frac{1}{2}})^{\nu}A^{\frac{1}{2}}.$$ In
addition,  the Heinz mean  of $ A $ and $ B $ is  defined
  as
$$H_{\nu} (A, B)=\dfrac{A\sharp_{\nu}B+ A\sharp_{1-\nu}B}{2}. $$
See \cite{ fu2,  fu1}  for more information about these means.

Using the above notations, the operator versions of Young
inequality, its refinements and its reverses are proved. For
instance, we have the following refinement of \eqref{1.7} and
\eqref{1.8} is obtained in \cite{jmi}. The other inequalities are
in similar way.
  \begin{theorem}\cite{jmi}
  Let $ A, B  \in  \mathbb{B}(\mathbb{H}) $ be positive invertible operators
  and positive real numbers $ m, m ', M, M' $ satisfy either $ 0 < m'I \leq A
  \leq mI < MI \leq B \leq M'I $ or  $ 0 < m'I \leq B  \leq mI < MI \leq A \leq M'I. $

  (I) If $ 0 < \nu \leq \frac{1}{2}, $ then
  \begin{equation*}
  A\nabla_{\nu}B \geq 2\nu (A\nabla B- A\sharp B) + r( A\sharp B -2 A\sharp_{
  \frac{1}{4}} B+A)+ K(h^{\frac{1}{4}} , 2)^{r_{1}}A\sharp_{\nu} B,
  \end{equation*} and
 \begin{equation*}
  A\nabla_{\nu}B \leq 2(1-\nu) (A\nabla B- A\sharp B) - r( A\sharp B -2 A\sharp_{
  \frac{3}{4}} B+B)+ K(h^{\frac{1}{4}} , 2)^{-r_{1}}A\sharp_{\nu} B,
  \end{equation*}
   (II) If $ \frac{1}{2} < \nu <1 , $ then
  \begin{equation*}
  A\nabla_{\nu}B \geq 2(1-\nu) (A\nabla B- A\sharp B) + r( A\sharp B -2 A\sharp_{
  \frac{3}{4}} B+B)+ K(h^{\frac{1}{4}} , 2)^{r_{1}}A\sharp_{\nu} B,
  \end{equation*} and
 \begin{equation*}
  A\nabla_{\nu}B \leq 2\nu (A\nabla B- A\sharp B) - r( A\sharp B -2 A\sharp_{
  \frac{1}{4}} B+A)+ K(h^{\frac{1}{4}} , 2)^{-r_{1}}A\sharp_{\nu} B,
  \end{equation*}
  where $ h=\frac{M}{m}, r=\min \{ \nu, 1-\nu \}$ and  $r_{1}=\min \{ 2r, 1-2r \}$.
  \end{theorem}

The main aim of this paper, is to state a generalization of these
inequalities. First, we present some generalizations of numerical
inequalities and base of them we prove some refined  operator
versions of Young inequality and its reverse. Also some
inequalities for Hilbert-Schmidt norm of matrices are obtained.

Throughout, for $0\leq\nu\leq1$, the notations $m_k=  \lfloor
2^k\nu \rfloor $ is the largest integer not greater than
 $2^k\nu$,  $ r_0=\min \{ \nu,
1-\nu\},  $   $ r_{k}=\min \{ 2r_{k-1}, 1-2r_{k-1} \},  $    for
$k\geq1$ and $R_k=1-r_k$.
\section{Numerical results}
Our first theorem, states a refined version of Young inequality
and its reverse.
\begin{theorem}\label{l01}
Let $ a, b $ be two positive real numbers and $ \nu \in [0, 1]. $
Then
\begin{align}\label{y1}
 K(h^{\frac{1}{2^n}},2)^{r_n} a\sharp_{\nu}b &\leqslant a\nabla_{\nu} b -
 \sum_{k=0}^{n-1}r_{k}\big[\big(a^{1-\frac{m_k}{2^k}}b^{\frac{m_k}{2^k}}
\big)^{\frac{1}{2}}-\big(a^{1-\frac{m_k+1}{2^k}}b^{\frac{m_k+1}{2^k}}
\big)^{\frac{1}{2}}\big]^{2}\\
 &\leqslant K(h^{\frac{1}{2^n}},2)^{R_n} a\sharp_{\nu}b,\nonumber
\end{align}
where $h=\frac{b}{a}$.

In addition, if $\nu=\frac{p}{2^{t}}$ for some $p, t\in\mathbb{N}$
with $t>1$, then
\begin{align*}
 K(h^{\frac{1}{2^{t-1}}},2)^{r_{t-1}} a\sharp_{\nu}b &= a\nabla_{\nu} b -
 \sum_{k=0}^{t-2}r_{k}\big[\big(a^{1-\frac{m_k}{2^k}}b^{\frac{m_k}{2^k}}
\big)^{\frac{1}{2}}-\big(a^{1-\frac{m_k+1}{2^k}}b^{\frac{m_k+1}{2^k}}
\big)^{\frac{1}{2}}\big]^{2}\\
 &= K(h^{\frac{1}{2^{t-1}}},2)^{R_{t-1}} a\sharp_{\nu}b,\nonumber
\end{align*}
\end{theorem}

\begin{proof}
First, we prove the left hand of inequality (\ref{y1}),   by
induction. For $n=1$, we get to the inequality \eqref{e03}. Let
inequality \eqref{y1} holds for $n$.

For $ 0 < \nu \leqslant \frac{1}{2}$,  we have
\begin{align*}
 a\nabla_{\nu} b -r_0(\sqrt{a}-\sqrt{b})^2&
 = a\nabla_{\nu} b -\nu(\sqrt{a}-\sqrt{b})^2\\
&=2\nu \sqrt{ab}+(1-2\nu)a\\
&=a\nabla_{2\nu}\sqrt{ab}
\end{align*}
Applying inequality \eqref{y1} for two positive numbers $a$ and
$\sqrt{ab}$ and $2\nu\in(0,1]$, we have
\begin{align*}
 a\nabla_{\nu} b
 -r_0(\sqrt{a}-\sqrt{b})^2&=a\nabla_{2\nu}\sqrt{ab}\\
 &\geq K(\sqrt{h}^{\frac{1}{2^n}},2)^{r_{n+1}}a\sharp_{2\nu}\sqrt{ab}+
 \sum_{k=0}^{n-1}r_{k+1}\big[\big(a^{1-\frac{m_{k+1}}{2^k}}(\sqrt{ab})^{
 \frac{m_{k+1}}{2^k}}
\big)^{\frac{1}{2}} \\&\quad
-\big(a^{1-\frac{m_{k+1}+1}{2^k}}(\sqrt{ab})^{\frac{m_{k+1}+1}{2^k}}
\big)^{\frac{1}{2}}\big]^{2}\\
&=K(h^{\frac{1}{2^{n+1}}},2)^{r_{n+1}}a\sharp_{\nu}b+
\sum_{k=1}^{n}r_{k}\big[\big(a^{1-\frac{m_k}{2^k}}b^{\frac{m_k}{2^k}}
\big)^{\frac{1}{2}}-\big(a^{1-\frac{m_k+1}{2^k}}b^{\frac{m_k+1}{2^k}}
\big)^{\frac{1}{2}}\big]^{2}.
\end{align*}

For $\frac{1}{2}<\nu<1$, we can apply the first part for $1-\nu$
and replace $a$ and $b$. Note that
$\lfloor2^k(1-\nu)\rfloor=2^k-\lfloor2^k\nu\rfloor-1$, if $2^k\nu$
is not integer. Thus, if $2^k\nu$ is not integer for each $k$, the
inequality follows.

Now, let $\nu=\frac{p}{2^{q}}$ for some $q>1$ and odd number $p$.
Since for each $i<q$, the coefficient $r_i\leq\frac{1}{2}$ is of
the form $\frac{p_i}{2^{q-i}}$, for some odd number $q_i$, it can
be concluded that $r_{q-1}=\frac{1}{2}=R_{q-1}$.  So the equality
follows.

A similar argument, leads to the second inequality.
\end{proof}
Changing the elements $a$ and $b$ in inequality \eqref{y1}, we can
state the following result for Heinz mean.
\begin{corollary}
Let $ a, b $ be two positive real numbers and $ \nu \in (0, 1). $
Then
\begin{align*}
 K(h^{\frac{1}{2^n}},2)^{r_n}H_{\nu}(a,b)
& \leqslant  a\nabla b -\sum_{k=0}^{n-1}r_{k}\big[H_{\frac{m_k}{2^k}}(a,b)-2H_{\frac{2m_k+1}{2^{k+1}}}(a,b)
+H_{\frac{m_k+1}{2^k}}(a,b)\big]\\
 &\leqslant K(h^{\frac{1}{2^n}},2)^{R_n}H_{\nu}(a,b),
\end{align*}
where $h=\frac{b}{a}$.
\end{corollary}


In the  following theorem, we state another version of the reverse
of Young inequality.
\begin{theorem}\label{l2}
Let $ a, b $ be two positive real numbers and $ \nu \in (0, 1). $
Then
\begin{equation}\label{y2}
 a\nabla_{\nu} b \leqslant K(h^{\frac{1}{2^n}},2)^{-r_n} a\sharp_{\nu}b+(
 \sqrt{a}-\sqrt{b})^2-
\sum_{k=0}^{n-1}r_{k}\big[\big(a^{\frac{m_k}{2^k}}b^{1-\frac{m_k}{2^k}}
\big)^{\frac{1}{2}}-\big(a^{\frac{m_k+1}{2^k}}b^{1-\frac{m_k+1}{2^k}}
\big)^{\frac{1}{2}}\big]^{2},
\end{equation}
where $h=\frac{b}{a}$.
\end{theorem}

\begin{proof}
Applying arithmetic-geometric mean inequality we have
$$K(h^{\frac{1}{2^n}},2)^{-r_n}a\sharp_{\nu}b+K(h^{\frac{1}{2^n}},2)^{r_n}b
\sharp_{\nu}a\geq2\sqrt{ab}.$$ Using this inequality and applying
inequality \eqref{y1}, we have
\begin{align*}
(\sqrt{a}-\sqrt{b})^2-a\nabla_{\nu}b&=b\nabla_{\nu}a-2\sqrt{ab}
-K(h^{\frac{1}{2^n}},2)^{-r_n}a\sharp_{\nu}b+\sum_{k=0}^{n-1}r_{k}\big[\big
(a^{\frac{m_k}{2^k}}b^{1-\frac{m_k}{2^k}}
\big)^{\frac{1}{2}}-\big(a^{\frac{m_k+1}{2^k}}b^{1-\frac{m_k+1}{2^k}}
\big)^{\frac{1}{2}}\big]^{2}.
\end{align*}
So the result follows.
\end{proof}

\begin{corollary}
Let $ a, b $ be two positive real numbers and $ \nu \in (0, 1). $
Then
\begin{equation*}
 a\nabla b \leqslant K(h^{\frac{1}{2^n}},2)^{-r_n}H_{\nu}(a,b)+(\sqrt{a}-\sqrt{b})^2-
\sum_{k=0}^{n-1}r_{k}\big[H_{\frac{m_k}{2^k}}(a,b)-2H_{\frac{2m_k+1}{2^{k+1}}}(a,b)
+H_{\frac{m_k+1}{2^k}}(a,b)\big],
\end{equation*}
where $h=\frac{b}{a}$.
\end{corollary}

\begin{remark}
Replacing $ a $ and $ b $ by their squares in (\ref{y1}) and (\ref{y2}), respectively, we obtain
\begin{align}\label{y3}
 K(h^{\frac{1}{2^{n-1}}},2)^{r_n}a^{2}\sharp_{\nu}b^{2} & \leqslant a^{2}\nabla_{\nu} b^{2}  -
 \sum_{k=0}^{n-1}r_{k} \big[  a^{1-\frac{m_k}{2^k}}b^{\frac{m_k}{2^k}}
  - a^{1-\frac{m_k+1}{2^k}}b^{\frac{m_k+1}{2^k}}
   \big]^{2}\\
   & \leqslant K(h^{\frac{1}{2^{n-1}}},2)^{R_n}a^{2}\sharp_{\nu}b^{2} \nonumber
\end{align}
and

\begin{equation}\label{y4}
 a^{2}\nabla_{\nu} b^{2} \leqslant K(h^{\frac{1}{2^{n-1}}},2)^{-r_{n}}
 a^{2}\sharp_{\nu}b^{2}+(a - b)^2-
\sum_{k=0}^{n-1}r_{k}\big[a^{\frac{m_k}{2^k}}b^{1-\frac{m_k}{2^k}}
- a^{\frac{m_k+1}{2^k}}b^{1-\frac{m_k+1}{2^k}} \big]^{2},
\end{equation}
where $h=\frac{b}{a}$.
\end{remark}
The following two theorems, are useful to prove a version of these
inequalities for the Hilbert-Schmidt norm of matrices.
\begin{theorem}\label{l3}
Let $ a, b $ be two positive real numbers and $ \nu \in (0, 1). $
Then
\begin{align}\label{y5}
K(h^{\frac{1}{2^{n-1}}},2)^{r_{n}}( a \sharp_{\nu}b )^{2}  & \leqslant
( a \nabla_{\nu} b)^{2}  - r_{0}^{2}(a-b)^{2}-
 \sum_{k=1}^{n-1}r_{k} \big[  a^{1-\frac{m_k}{2^k}}b^{\frac{m_k}{2^k}}
  - a^{1-\frac{m_k+1}{2^k}}b^{\frac{m_k+1}{2^k}}
   \big]^{2}\nonumber\\
   & \leqslant K(h^{\frac{1}{2^{n-1}}},2)^{R_{n}}( a \sharp_{\nu}b )^{2},
\end{align}
where $h=\frac{b}{a}$.
\end{theorem}
\begin{proof}

By a simple calculation, we have $ ( a \nabla_{\nu} b)^{2} -
r_{0}^{2}(a
-b)^{2} =a^{2} \nabla_{\nu} b^{2}-r_{0} (a-b)^{2}. $\\
 Using (\ref{y3}), we have
\begin{align*}
K(h^{\frac{1}{2^{n-1}}},2)^{r_n}(a \sharp_{\nu}b)^{2} & \leqslant a^{2}
\nabla_{\nu} b^{2}  -
 \sum_{k=0}^{n-1}r_{k} \big[  a^{1-\frac{m_k}{2^k}}b^{\frac{m_k}{2^k}}
  - a^{1-\frac{m_k+1}{2^k}}b^{\frac{m_k+1}{2^k}}
   \big]^{2}\\
   &= (a \nabla_{\nu} b)^{2}   -  r_{0}^{2}(a-b)^{2}
- \sum_{k=1}^{n-1}r_{k} \big[  a^{1-\frac{m_k}{2^k}}b^{\frac{m_k}{2^k}}
  - a^{1-\frac{m_k+1}{2^k}}b^{\frac{m_k+1}{2^k}}
   \big]^{2}\\
   & \leqslant K(h^{\frac{1}{2^{n-1}}},2)^{R_n}(a\sharp_{\nu}b)^{2}
\end{align*}
\end{proof}
\begin{theorem}\label{l4}
Let $ a, b $ be two positive real numbers and $ \nu \in (0, 1). $
Then
\begin{equation}\label{y6}
 ( a \nabla_{\nu} b)^{2} \leqslant K(h^{\frac{1}{2^{n-1}}},2)^{-r_{n}}( a
 \sharp_{\nu}b )^{2}+ R_{0}^{2}(a-b)^{2}-
 \sum_{k=1}^{n-1}r_{k} \big[  a^{1-\frac{m_k}{2^k}}b^{\frac{m_k}{2^k}}
  - a^{1-\frac{m_k+1}{2^k}}b^{\frac{m_k+1}{2^k}}
   \big]^{2},
\end{equation}
where $h=\frac{b}{a}$.
\end{theorem}
\begin{proof}
  We have
\begin{align*}
( a \nabla_{\nu} b)^{2} & - (1- r_{0})^{2}(a-b)^{2} \\
&=a^{2} \nabla_{\nu} b^{2}-(1-r_{0}) (a-b)^{2}\\
& \leqslant  K(h^{\frac{1}{2^{n-1}}},2)^{-r_{n}}( a \sharp_{\nu}
b)^{2}+r_{0}(a-b)^{2}  - \sum_{k=0}^{n-1}r_{k} \big[
a^{1-\frac{m_k}{2^k}}b^{\frac{m_k}{2^k}}
  - a^{1-\frac{m_k+1}{2^k}}b^{\frac{m_k+1}{2^k}}
   \big]^{2}\\
   & \ \ \ \ \  \ \ \ \ \ \ \ \ \ \ \text{ by  inequality (\ref{y4})}\\
   &=K(h^{\frac{1}{2^{n-1}}},2)^{-r_{n}}( a \sharp_{\nu} b)^{2}  -\sum_{k=1}^{\infty}r_{k} \big[  a^{1-
   \frac{m_k}{2^k}}b^{\frac{m_k}{2^k}}
     - a^{1-\frac{m_k+1}{2^k}}b^{\frac{m_k+1}{2^k}}
      \big]^{2}.\\
\end{align*}
\end{proof}

\section{Related operator inequalities}
To state the operator versions of the inequalities obtained in
section 2, we need the following lemma.
\begin{lemma}\cite{f3}\label{l1}
Let $ X \in  \mathbb{B}(\mathbb{H}) $ be self-adjoint  and let $ f $ and $ g $
be continuous
real functions such that $ f(t)  \geqslant g(t)$ for all $ t \in \sigma(X) $
(the spectrum of $ X $). Then $ f(X) \geqslant g(X). $
\end{lemma}
Let $X$ be a strictly positive operator. Then $\sigma(X)$ is a
compact subset of  $(0,+\infty)$. We denote by $m(X)$ and $M(X)$
the minimum and the maximum of $\sigma(X)$.

Now, we give the  first result in this section which is based on
Theorem \ref{l01} and is a refinement of Theorem 3 in \cite{jmi}.
\begin{theorem}\label{t1}
 Let $A,B\in \mathbb{B}(H)$ be two
strictly positive operators with $M(A)\leq m(B)$ and $ \nu \in (0,
1). $
\begin{align}\label{1e}
 K(h^{\frac{1}{2^{n}}},2)^{r_{n}} A \sharp_{\nu} B & \leqslant  A \nabla_{\nu} B -
\sum_{k=0}^{n-1}r_{k} [ A\sharp_{\frac{m_k}{2^k}}B -2
A\sharp_{\frac{2m_k+1 }{2^{k+1}}}B +
A\sharp_{{\frac{m_k+1}{2^k}}}B]\\
 &\leqslant K(h^{\frac{1}{2^{n}}},2)^{R_{n}} A \sharp_{\nu} B, \nonumber
\end{align}
where $h=\frac{m(B)}{M(A)}$.
\end{theorem}
\begin{proof}
Choosing $ a=1, $ in Theorem \ref{l01},  we have
\begin{equation*}
 1-\nu+ \nu b \geqslant K(b^{\frac{1}{2^{n}}},2)^{r_{n}} b^{\nu}+
\sum_{k=0}^{n-1}r_{k}\big[\big(b^{\frac{m_k}{2^k}}
\big)^{\frac{1}{2}}-\big(b^{\frac{m_k+1}{2^k}}
\big)^{\frac{1}{2}}\big]^{2},
\end{equation*}
for any $ b > 0. $\\
If    $ X=A^{-\frac{1}{2}}BA^{-\frac{1}{2}}, $  then  $\sigma(X)
\subseteq [h, +\infty). $ Due to Kantorovich constant is
increasing on $[1,+\infty)$, it follows that for all $b\geq h$,
\begin{align*}
 1-\nu+ \nu b &\geqslant K(b^{\frac{1}{2^{n}}},2)^{r_{n}} b^{\nu}+
\sum_{k=0}^{n-1}r_{k}\big[\big(b^{\frac{m_k}{2^k}}
\big)^{\frac{1}{2}}-\big(b^{\frac{m_k+1}{2^k}}
\big)^{\frac{1}{2}}\big]^{2}\\
&\geqslant K(h^{\frac{1}{2^{n}}},2)^{r_{n}} b^{\nu}+
\sum_{k=0}^{n-1}r_{k}\big[\big(b^{\frac{m_k}{2^k}}
\big)^{\frac{1}{2}}-\big(b^{\frac{m_k+1}{2^k}}
\big)^{\frac{1}{2}}\big]^{2},
\end{align*}

According to  Lemma \ref{l1}, we get
\begin{equation}\label{e1}
 (1-\nu)I+\nu X \geqslant K(h^{\frac{1}{2^{n}}},2)^{r_{n}} X^{\nu}+
\sum_{k=0}^{n-1}r_{k} [X^{\frac{m_k}{2^k}} -2X^{\frac{2m_k+1
}{2^{k+1}}} + X^{\frac{m_k+1}{2^k}}].
\end{equation}
Multiplying both sides of (\ref{e1}) by $ A^{\frac{1}{2}},$ we
obtain
\begin{equation*}
 A \nabla_{\nu} B \geqslant K(h^{\frac{1}{2^{n}}},2)^{r_{n}}  A \sharp_{\nu} B+
\sum_{k=0}^{n-1}r_{k} [ A\sharp_{\frac{m_k}{2^k}}B -2
A\sharp_{\frac{2m_k+1 }{2^{k+1}}}B +
A\sharp_{{\frac{m_k+1}{2^k}}}B].
\end{equation*}
This completes the proof of left hand of inequality (\ref{1e}), by the same way, we can prove the right hand.
\end{proof}

The following theorem is an operator version of Theorem \ref{l2} and is a refinement
of Theorem 4 in \cite{jmi}.
\begin{theorem}
Let $ A, B \in  \mathbb{B}(\mathbb{H}) $ be two positive
invertible operators with $M(A)\leq m(B)$ and $ \nu \in (0, 1). $
\begin{equation*}
 A \nabla_{\nu} B \leqslant K(h^{\frac{1}{2^{n}}},2)^{-r_{n}}  A \sharp_{\nu} B +
 (A-2A \sharp B + B  ) -
\sum_{k=0}^{n-1}r_{k} [ A\sharp_{\frac{m_k}{2^k}}B -2
A\sharp_{\frac{2m_k+1 }{2^{k+1}}}B +
A\sharp_{{\frac{m_k+1}{2^k}}}B],
\end{equation*}
where $h=\frac{m(B)}{M(A)}$.
\end{theorem}
\begin{proof}
By Lemma \ref{l2}, using the same ideas as in the proof of Theorem
\ref{t1}, we can get this theorem.
\end{proof}
\begin{corollary}
Let $ A, B \in  \mathbb{B}(\mathbb{H}) $ be two positive
invertible operators with $M(A)\leq m(B)$ and $ \nu \in (0, 1). $
Then
\begin{align*}
 K(h^{\frac{1}{2^{n}}},2)^{r_{n}} H_{\nu}(A, B) & \leqslant  A \nabla B
-\sum_{k=0}^{n-1}r_{k} [ H_{\frac{m_k}{2^k}}(A,B) -2
H_{\frac{2m_k+1 }{2^{k+1}}}(A,B) + H_{{\frac{m_k+1}{2^k}}}(A,B)]\\
& \leqslant K(h^{\frac{1}{2^{n}}},2)^{R_{n}} H_{\nu}(A, B)
\end{align*}
and
\begin{equation*}
 A \nabla B \leqslant K(h^{\frac{1}{2^{n}}},2)^{-r_{n}} H_{\nu}(A, B) +(A-2A \sharp B + B  ) -
\sum_{k=0}^{n-1}r_{k} [ H_{\frac{m_k}{2^k}}(A,B) -2
H_{\frac{2m_k+1 }{2^{k+1}}}(A,B) + H_{{\frac{m_k+1}{2^k}}}(A,B)],
\end{equation*}
where $h=\frac{m(B)}{M(A)}$.
\end{corollary}

\section{Matrix Young and reverse inequalities for the Hilbert-Schmidt norm}
In this  section, we present some inequalities for the Hilbert-Schmidt norm.
It is known that every positive semidefinite matrix is unitarily diagonalizable.
Thus for two positive semidefinite  $ n\times n $ matrices $ A $ and  $ B, $
there  exist two  unitary matrices $ U $ and  $ V $  such that
 $ A=U diag (\lambda_{1}, \ldots, \lambda_{n})U^{*} $ and    $ B=V diag (\mu_{1},
 \ldots, \mu_{n})V^{*}. $\\
Applying Theorem \ref{l3}, we get the following theorem that is  a
refinement of the inequalities in  \cite[Theorem 5]{jmi}.
\begin{theorem}\label{t2}
Suppose
$ A, B, X \in  \mathbb{M}_{n}$  such that $ A $ and  $ B$  are
two  positive definite matrices and $ \nu \in (0,  1). $  Let
$$ \underline{K _{t}}= \min  \left \lbrace K \left(( \frac{\mu_{j}}{\lambda_{i}})^{
\frac{1}{2^{t-1}}},2 \right)^{r_{t}}:\   i, j=1,2, \ldots, n
\right \rbrace,  $$ and
$$ \overline{K _{t}}= \max  \left \lbrace K \left(( \frac{\mu_{j}}{\lambda_{i}})^{
\frac{1}{2^{t-1}}},2 \right)^{R_{t}}:\   i, j=1,2, \ldots, n
\right \rbrace,  $$
  for all $ t \in \mathbb{N}. $ Then
\begin{align}\label{fe}
\underline{K _{t}} \|
A^{1-\nu}XB^{\nu}\|_{2}^{2}  &\leqslant \|(1-\nu)AX-\nu XB \|_{2}^{2} -r_{0}^{2}\|AX-XB\|_{2}^{2} \nonumber \\
& \quad -\sum _{k=1}^{t-1}r_{k}
\|A^{1-\frac{m_{k}}{2^{k}}}XB^{\frac{m_{k}}{2^{k}}}-
A^{1-\frac{m_{k}+1}{2^{k}}}X
B^{\frac{m_{k}+1}{2^{k}}} \|_{2}^{2}  \nonumber\\
&\leqslant \overline{K _{t}} \|
A^{1-\nu}XB^{\nu}\|_{2}^{2}  .
\end{align}
\end{theorem}
\begin{proof}
Let $ Y=U^{*}XV=(y_{ij}). $ Then
$$ (1-\nu)AX-\nu XB =U [((1-\nu)\lambda_{   i }+\nu \mu_{j})\circ Y] V^{*}, $$
 $$ AX-XB=U[(\lambda_{i}-\mu_{j})\circ Y] V^{*} $$
$$ A^{1-\nu} X B^{\nu}=U [(\lambda_{i}^{1- \nu}\mu_{j}^{\nu})\circ Y]V^{*}$$   and
$$A^{1-\frac{m_{k}}{2^{k}}}XB^{\frac{m_{k}}{2^{k}}}- A^{1-\frac{m_{k}+1}{2^{k}}}X
B^{\frac{m_{k}+1}{2^{k}}}=U[(\lambda_{i}^{1-\frac{m_{k}}{2^{k}}}
\mu_{j}^{\frac{m_{k}}{2^{k}}}- \lambda_{i}^{1-\frac{m_{k}}{2^{k}}}
\mu_{j}^{\frac{m_{k}}{2^{k}}}) \circ Y]V^{*}.$$ Utilizing the
unitarily invariant  property of $ \|.\|_{2} $ and Theorem
\ref{l3}, we have
\begin{align*}
 &\|(1-\nu)AX-\nu XB \|_{2}^{2}\\
 &=\sum_{i,j=1}^{n} ((1-\nu)\lambda_{ i }+\nu \mu_{j})^2|y_{ij}|^{2}\\
 & \geqslant\sum_{i,j=1}^{n} \left \{K \left(( \frac{\mu_{j}}{\lambda_{i}})^{\frac{1}{2^{t-1}}},2 \right)^{r_{t}}
 (\lambda_{i}^{1- \nu}\mu_{j}^{\nu})^2 + r_{0}^{2}
 (\lambda_{i}-\mu_{j})^{2}+ \sum _{k=1}^{t-1}r_{k}(\lambda_{i}^{1-\frac{m_{k}}{2^{k}}}
 \mu_{j}^{\frac{m_{k}}{2^{k}}}-
 \lambda_{i}^{1-\frac{m_{k}+1}{2^{k}}} \mu_{j}^{\frac{m_{k}+1}{2^{k}}})^{2}    \right\}|y_{ij}|^{2}\\
 &=\sum_{i,j=1}^{n}K \left(( \frac{\mu_{j}}{\lambda_{i}})^{\frac{1}{2^{t-1}}},2 \right)^{r_{t}}(
 \lambda_{i}^{1- \nu}\mu_{j}^{\nu})^2|y_{ij}|^{2}+ \sum_{i,j=1}^{n}  r_{0}^{2}
 (\lambda_{i}-\mu_{j})^{2}|y_{ij}|^{2}\\
 & \quad +   \sum_{i,j=1}^{n} \left   \{   \sum _{k=1}^{t-1}r_{k}(\lambda_{i}^{1-
 \frac{m_{k}}{2^{k}}} \mu_{j}^{\frac{m_{k}}{2^{k}}}-
  \lambda_{i}^{1-\frac{m_{k}+1}{2^{k}}} \mu_{j}^{\frac{m_{k}+1}{2^{k}}})^{2}  |y_{ij}|^{2}    \right\}\\
  &\geq \underline{K _{t}}\sum_{i,j=1}^{n}(\lambda_{i}^{1- \nu}\mu_{j}^{\nu})^2|y_{ij}|^{2}+ \sum_{i,j=1}^{n}
  r_{0}^{2}(\lambda_{i}-\mu_{j})^{2}|y_{ij}|^{2}\\
  & \quad +   \sum_{k=1}^{t-1} \left   \{   \sum _{i, j=1}^{n}r_{k}
  (\lambda_{i}^{1-\frac{m_{k}}{2^{k}}} \mu_{j}^{\frac{m_{k}}{2^{k}}}-
   \lambda_{i}^{1-\frac{m_{k}+1}{2^{k}}} \mu_{j}^{\frac{m_{k}+1}{2^{k}}})^{2}  |y_{ij}|^{2}    \right\}\\
   &= \underline{K _{t}} \| A^{1-\nu}XB^{\nu}\|_{2}^{2}+r_{0}^{2}\|AX-XB\|_{2}^{2} \\
   & \quad +\sum _{k=1}^{t-1}r_{k} \|A^{1-\frac{m_{k}}{2^{k}}}XB^{\frac{m_{k}}{2^{k}}}
   - A^{1-\frac{m_{k}+1}{2^{k}}}XB^{\frac{m_{k}+1}{2^{k}}} \|_{2}^{2}.
\end{align*}
This complete the  proof of the left side of (\ref{fe}). By the same ideas, we can prove the right side.
\end{proof}

\begin{theorem}
Suppose  $ A, B, X \in  \mathbb{M}_{n}$  such that $ A $ and  $ B$  are
  positive definite matrices  and
$ \nu \in (0, 1). $ Let
$$ \underline{K _{t}}= \min  \left \lbrace K \left(( \frac{\mu_{j}}{\lambda_{i}})^{
\frac{1}{2^{t-1}}},2 \right)^{r_{t}}:\   i, j=1,2, \ldots, n
\right \rbrace  $$.
 Then
\begin{align*}
 \|(1-\nu)AX-\nu XB \|_{2}^{2} & \leqslant  \underline{K _{t}}^{-1}\| A^{1-\nu}XB^{\nu}\|_{2}^{2}+R_{0}^{2}
 \|AX-XB\|_{2}^{2}\\
& \quad -\sum _{k=1}^{\infty}r_{k} \|A^{1-\frac{m_{k}}{2^{k}}}XB^{\frac{m_{k}}{2^{k}}}-
 A^{1-\frac{m_{k}+1}{2^{k}}}XB^{\frac{m_{k}+1}{2^{k}}} \|_{2}^{2}.
\end{align*}
\end{theorem}
\begin{proof}
By Theorem \ref{l4}, using  the same idea as in the proof of Theorem
\ref{t2}, we can obtain the desired result.
\end{proof}
\bibliographystyle{amsplain}

\end{document}